\documentclass[]{amsart}
\usepackage{hyphenat}
\usepackage[T1]{fontenc}
\usepackage{amsmath,amssymb,mathtools,amsthm}
\usepackage{microtype}
\usepackage{orcidlink}

\newtheorem{thm}{Theorem}[section]

\newtheorem*{fixed point criterion}{\fixed point criterion}
\newtheorem{cor}[thm]{Corollary}
\newtheorem{lem}[thm]{Lemma}
\newtheorem{prop}[thm]{Proposition}

\newtheorem{conj}[thm]{Conjecture}
\theoremstyle{definition}
\newtheorem{defn}[thm]{Definition}

\newtheorem{ques}[thm]{Question}

\theoremstyle{remark}

\numberwithin{equation}{section}

\newcommand{\BS}{\operatorname{BS}}
\newcommand{\Area}{\operatorname{Area}}
\newcommand{\id}{\operatorname{id}}
\newcommand{\ncl}[1]{\left\langle\!\left\langle #1\right\rangle\!\right\rangle}

\title[One-relator groups with a commutator relator]
{Commutator relators of one-relator groups do not force Hopficity, residual finiteness, or automaticity}

\author[Ke Wang]{Ke Wang\orcidlink{0009-0005-7108-3725}}
\address{School of Mathematics and Statistics, Xi'an Jiaotong University, Xi'an 710049, P. R. China}
\email{keqiyehuopo@stu.xjtu.edu.cn}
	
\author[Qiang Zhang]{Qiang Zhang\orcidlink{0000-0001-6332-5476}}
\address{School of Mathematics and Statistics, Xi'an Jiaotong University, Xi'an 710049, P. R. China}
\email{zhangq.math@mail.xjtu.edu.cn}


\date{\today}

\thanks{The authors are partially supported by NSFC (No. 12471066).}

\begin{document}

\begin{abstract}
Let $G=F/\ncl{r}$ be a one-relator group with the relator $r\in [F,F]$ or $r=[u,v] ~(u,v\in F)$, where $F$ is a finitely generated free group. Baumslag asked whether $G$ is Hopfian, residually finite or automatic. In the case of $r\in[F,F]$, a negative answer to the residual finiteness and automaticity has already been obtained by a result of Olshanskii. In this note, we construct a family of one-relator groups
$$ G_m=\left\langle a,t\ \middle|\ [t,a[a,t]^{-m}]\right\rangle,$$
 whose relators are commutators, each of which has a Baumslag-Solitar subgroup as a retract. These groups provide negative answers to these three questions in both cases. 
\end{abstract}

\maketitle

\section{Introduction}

One-relator groups have been one of central subjects in combinatorial and geometric group theory since the pioneering work of Magnus. His Freiheitssatz and solution of the word problem established the basic inductive methods of the theory \cite{Ma30, Ma32}. For one-relator groups with torsion, Newman \cite{Ne68} later proved his influential spelling theorem, which yields strong control over their normal forms and implies their hyperbolicity. Building on the development of geometric group theory, Wise \cite{Wi21} proved that one-relator groups with torsion are virtually special and hence residually finite and linear, resolving a longstanding conjecture of Baumslag. Further structural progress was made by Louder and Wilton, who proved that one-relator groups with torsion are coherent \cite{LW20} and subsequently developed the theory of negative immersions to establish coherence and related subgroup properties for a broad class of torsion-free one-relator groups \cite{LW24}. Most recently, Jaikin-Zapirain and Linton \cite{JL25} proved that every one-relator group is coherent, thereby settling the general coherence conjecture. Their work also gives homological coherence results for group algebras and extends to wider classes of groups of cohomological dimension two. Together, these results illustrate how classical combinatorial arguments, geometric methods, and modern homological techniques have progressively deepened the understanding of one-relator groups.

In this note, we focus on one-relator groups whose relator is a commutator, and provide a solution to the following open questions of Baumslag.
\begin{ques}\cite[Questions OR7, OR8]{BMS98}\label{ques. OR8}
Let $G=F/\ncl{r}$ be a one-relator group, where $F$ is a finitely generated free group and $r\in [F, F]$ or $r=[u, v]$ for some $u,v\in F$.
\begin{enumerate}
    \item[(a)] Is $G$ Hopfian?
    \item[(b)] Is $G$ residually finite?
    \item[(c)] Is $G$ automatic?
\end{enumerate}
\end{ques}

For items (b) and (c) for the case of $r\in[F,F]$, Baumslag said that there exists such one-relator groups that are neither residually finite nor automatic, which can be obtained from a result of Olshanskii \cite{Ol95}. To our knowledge, there are no known counterexamples or affirmative proofs for item (a) for the case of $r\in[F,F]$ or items (a), (b) and (c) for the case of $r=[u,v]$. Thus, the purpose of this note is to provide such counterexamples to these questions.

Recall that a group is \emph{Hopfian} if its every surjective endomorphism is injective, and it is \emph{residually finite} if the intersection of all its normal subgroups of finite index is trivial. The following definition of \emph{automatic group} is from \cite[Definition 2.3.1]{Ep92}. Refer to \cite[Chapter 2]{Ep92} for more details.
\begin{defn}[Automatic group]\label{defn: automatic group}
Let $G$ be a group. An \emph{automatic structure} on $G$ consists of a set $A$ of semigroup generators of $G$, a finite state automaton $W$ over $A$, and finite state automata $M_x$ over $(A,A)$, for $x\in A\cup\{\epsilon\}$, satisfying the following conditions:
\begin{enumerate}
    \item The map $\pi:L(W) \to G$ is surjective.
    \item For $x\in A\cup\{\epsilon\}$, we have $(w_1,w_2)\in L(M_x)$ if and only if $\overline{w_1x}=\overline{w_2}$ and both $w_1$ and $w_2$ are elements of $L(W)$.
\end{enumerate}
The symbol $\epsilon$ represents the null-string and $L(W)$ is the language recognized by $W$. We call $W$ the \emph{word acceptor}, $M_\epsilon$ the \emph{equality recognizer}, and each $M_x$, for $x \in A$, a \emph{multiplier automaton} for the automatic structure. An automatic structure is sometimes called an \emph{automation}. An \emph{automatic group} is one that admits an automatic structure.
\end{defn}

\noindent\textbf{Notations.} Throughout this paper, we use the commutator convention:
    $$ [x,y]:=x^{-1}y^{-1}xy,$$ 
    and denote the Baumslag-Solitar group as follows:
    $$\BS(p,q):= \left\langle c,t\ \middle|\ t^{-1}c^pt=c^q\right\rangle,\quad p,q\in\mathbb Z\backslash\{0\}.$$

Our main result is as follows.
\begin{thm}\label{thm:main}
For any integer $m\geq 1$, let
$$
 G_m=\left\langle a,t\ \middle|\ [t,a[a,t]^{-m}]\right\rangle .
$$
Then the following hold:
\begin{enumerate}
\item The group $G_m$ contains $\BS(m,m+1)$ as a retract.
\item The group $G_2$ is not Hopfian, and hence it is not residually finite.
\item The group $G_1$ is not automatic.
\end{enumerate}
\end{thm}

As a direct corollary, we have:

\begin{cor}\label{cor:questions}
For a one-relator group $F/\ncl{r}$, neither of the assumptions
$$ r=[u,v] \qquad\text{or}\qquad r\in[F,F]$$
implies Hopficity, residual finiteness, or automaticity.
\end{cor}

In Sect. \ref{sect. Gm}, we show that the one-relator group
$G_m$ each has a Baumslag-Solitar subgroup as a retract. In Sect. \ref{sect. G2}, we prove the non-Hopficity of $G_2$ by constructing an epimorphism of $G_2$, which induces a non-monomorphic epimorphism on its Baumslag-Solitar subgroup. Since every finitely generated residually finite group is Hopfian (a well known result of Mal'cev \cite{Ma40}), we can deduce that $G_2$ is not residually finite from its non-Hopficity. In Sect. \ref{sect. G1}, we prove the non-automaticity of $G_1$, by showing $G_1$ does not satisfy the quadratic isoperimetric inequality. Finally in Sect. \ref{sect. proof and question}, we complete the proof of our main results and pose some further questions.

\section{The group $G_m$ and its Baumslag-Solitar retracts}\label{sect. Gm}
For any integer $m\geq 1$, let 
\begin{equation}\label{eq:Gm}
G_m:=\left\langle a,t\ \middle|\ [t,a[a,t]^{-m}] \right\rangle.
\end{equation}
Denote $$c:=[a,t]=a^{-1}t^{-1}at.$$ Then $a^{-1}ta=tc^{-1},$ and the single relation $$[t,a[a,t]^{-m}]=1$$ is equivalent to
$$ t^{-1}[a, t]^ma^{-1}ta[a, t]^{-m}=t^{-1}c^mtc^{-(m+1)}=1,$$
namely, $t^{-1}c^mt=c^{m+1}.$ Now adding $c$ as a generator to Eq. (\ref{eq:Gm}), we obtain a new presentation of $G_m$ as follows:
\begin{equation}\label{eq:Gm-expanded}
 G_m\cong \left\langle a,c,t\ \middle|\ t^{-1}at=ac,\quad t^{-1}c^mt=c^{m+1} \right\rangle .
\end{equation}

\begin{prop}\label{prop:retract}
For every $m\geq 1$, the subgroup $\langle c,t\rangle$ of $G_m$ is naturally isomorphic to $\BS(m,m+1)$ and is a retract of $G_m$.
\end{prop}

\begin{proof}
Denote
$$ B_m:=\BS(m,m+1) =\left\langle c,t\ \middle|\ t^{-1}c^mt=c^{m+1}\right\rangle .$$
The presentation \eqref{eq:Gm-expanded} gives a homomorphism
$$ \iota_m:B_m\longrightarrow G_m, \qquad \iota_m(c)=c, \quad \iota_m(t)=t.$$
Furthermore, define a map on the generators of $G_m$ by
$$ \rho_m(a)=c^m, \qquad \rho_m(c)=c, \qquad \rho_m(t)=t.$$
Then the first relation in \eqref{eq:Gm-expanded} maps to
$$ t^{-1}c^mt=c^mc=c^{m+1}, $$
and the second maps to the defining relation of $B_m$. Hence $\rho_m$ induces a homomorphism $$ \rho_m:G_m\longrightarrow B_m.$$
Since $\rho_m\circ\iota_m=\id_{B_m}$, the map $\iota_m$ is injective and $B_m$ is a retract of $G_m$.
\end{proof}

We shall use the following Britton's lemma in the cyclic case; see, for example, \cite{Br63} or \cite[Page 181]{LS77} for the original statement. A brief corridor proof is included for completeness.

\begin{lem}[Britton's lemma in the cyclic case]\label{lem:britton}
Let $$ \BS(p,q)=\langle c,t\mid t^{-1}c^pt=c^q\rangle, \qquad p,q\neq 0.$$
Suppose that a word $w$ contains an occurrence of $t^{\pm1}$ and has no subword of either form
$$ t^{-1}c^{kp}t \qquad\text{or}\qquad tc^{kq}t^{-1}, \qquad k\in\mathbb Z.$$
Then $w$ does not represent the identity in $\BS(p,q)$.
\end{lem}

\begin{proof}
Assume that $w=1$ and choose a reduced van Kampen diagram for $w$ over the displayed presentation. The $t$-edges of the two-cells assemble into $t$-corridors. The usual innermost-annulus reduction shows that a reduced diagram contains no annular $t$-corridor. Hence every
$t$-corridor joins two boundary $t$-edges.

Choose an outermost corridor. One of the two boundary arcs cut off by it contains no other $t$-edge and is therefore labelled by a power $c^\ell$. According to the orientations of the terminal $t$-edges, the corresponding boundary subword is either
$$ t^{-1}c^\ell t \qquad\text{or}\qquad tc^\ell t^{-1}.$$
The corridor is tiled by copies of the defining two-cell. In the first case this forces $\ell\in p\mathbb Z$, and in the second it forces $\ell\in q\mathbb Z$. Thus $w$ contains a prohibited pinch, a contradiction.
\end{proof}

\section{\texorpdfstring{Non-Hopficity and non-residual finiteness of $G_2$}{Proof of non-Hopficity of $G_2$}}\label{sect. G2}

Now for $m=2$, the presentation \eqref{eq:Gm-expanded} becomes
\begin{equation}\label{eq:G2}
 G_2= \left\langle a,c,t\ \middle|\ t^{-1}at=ac,\quad t^{-1}c^2t=c^3 \right\rangle .
\end{equation}
Then we have:
\begin{prop}\label{prop:nonhopfian}
The group $G_2$ is not Hopfian, and hence not residually finite.
\end{prop}

\begin{proof}
Define a map on the generators of $G_2$ by
$$ \Phi(a)=ac^2, \qquad \Phi(c)=c^2, \qquad \Phi(t)=t.$$
For the first relation in \eqref{eq:G2}, we have
\begin{align*}
 \Phi(t^{-1}at)
 &=t^{-1}(ac^2)t \\
 &=(t^{-1}at)(t^{-1}c^2t) \\
 &=(ac)c^3 \\
 &=\Phi(ac),
\end{align*}
and for the second relation,
$$\Phi(t^{-1}c^2t) =t^{-1}c^4t =(t^{-1}c^2t)^2 =c^6 =\Phi(c^3). $$
Therefore $\Phi$ induces an endomorphism of $G_2$. Note that the image $\phi(G_2)$ contains $t$, $c^2$, and $ac^2$. It therefore contains
$$ a=(ac^2)c^{-2}=\Phi(ac^{-1}), \quad c^3=t^{-1}c^2t=\Phi(t^{-1}ct), \quad c=c^3(c^2)^{-1}=\Phi(t^{-1}ctc^{-1}).$$
Thus $\Phi: G_2\to G_2$ and $\Phi |_{B_2}: B_2\to B_2$ are both surjective, where $$B_2:=\langle c,t\mid t^{-1}c^2t=c^3\rangle\cong \BS(2,3),$$ by Proposition~\ref{prop:retract}.

We now consider
$$ w:=[t^{-1}ct,c]=t^{-1}c^{-1}tc^{-1}t^{-1}ctc\in B_2.$$
Note that the three possible intervening powers of $c$ have exponents $-1,-1,1$. None gives a pinch of the form $t^{-1}c^{2k}t$ or $tc^{3k}t^{-1}$. By Lemma~\ref{lem:britton}, $w\neq 1$ in $B_2$, and therefore also in $G_2$. On the other hand,
$$ \Phi(w)=[\Phi(t^{-1}ct),\Phi(c)]=[t^{-1}c^2t,c^2] =[c^3,c^2] =1.$$
Therefore, $\Phi$ is surjective but not injective. It follows that $G_2$ is not Hopfian and hence not residually finite.
\end{proof}

\section{\texorpdfstring{Non-automaticity of $G_1$}{Proof of the non-automaticity of $G_1$}}\label{sect. G1}

Let $\mathcal P=\langle X\mid R\rangle $ be a finite presentation, and let $F(X)$ be the free group generated by $X$. For a word $w\in F(X)$ representing the identity in $\mathcal P$, define 
$$ \Area_{\mathcal P}(w) := \min\left\{ N \mid w=\prod_{i=1}^{N}u_i r_i^{\varepsilon_i}u_i^{-1}\text{ in }F(X) \right\},$$
where $u_i\in F(X)$, $r_i\in R$, and $\varepsilon_i\in\{\pm1\}$. Equivalently, $\Area_{\mathcal P}(w)$ is the least number of 2-cells in a van Kampen diagram for $w$ over $\mathcal P$. Furthermore, we have the definition of \emph{Dehn function} $\delta:\mathbb N\to\mathbb N$:
$$\delta_{\mathcal P}(n):=\max\left\{\Area_{\mathcal P}(w) \mid |w|\leq n,~ w\in F(X)\mbox{ represents the identity in} ~\mathcal P \right\},$$
where $|w|$ denotes the word length of $w$ in $F(X)$, i.e., the number of letters in the freely reduced word representing $w$.

We shall use the following standard consequence of automaticity, see \cite[Theorem 2.3.12]{Ep92}.

\begin{thm}[\cite{Ep92}, Theorem 2.3.12]\label{thm:automatic-quadratic}
If a group $H$ is automatic, then every finite presentation
$\mathcal P$ of $H$ satisfies an inequality
$$ \Area_{\mathcal P}(w)\leq C|w|^2+C,\qquad \delta_{\mathcal P}(n)\leq Cn^2+C $$
for all null-homotopic words $w$ and $n\geq 1$, with a constant $C$ depending only on $\mathcal P$.
\end{thm}

The isoperimetric inequalities and the Dehn functions for Baumslag-Solitar groups indeed have been extensive studied, for example, \cite{Ep92,Ge89,Gr91}. The following theorem is due to S. M. Gersten \cite[Theorem B]{Ge89}. 

\begin{thm}[\cite{Ge89}, Theorem B]\label{thm: Gersten}
    If $X$ is the $2$-complex canonically associated to the Baumslag-Solitar presentation $\langle x, y \mid y^{-1}x^ky = x^\ell\rangle$ where $|k|\neq |\ell|$, then the Dehn function $\delta_X(n)$ grows faster than any polynomial function of $n$; that is, there are no constants $A,d > 0$ with $\delta_X(n) \leq An^d$ for all $n\geq 1$.
\end{thm}
Note that the original statement of Theorem \ref{thm: Gersten} is about a function $\lambda_X(n)$, which is not greater than the Dehn function $\delta_X(n)$. We also need the following result \cite[Corollary 2]{BMS93} due to Baumslag, Miller III and Short.
\begin{thm}[\cite{BMS93}, Corollary 2]\label{thm: Dehn function of retract}
If the finitely presented group $B$ with finite presentation $\mathcal{B}$ is a retract of the finitely presented group $A$ with finite presentation $\mathcal{A}$, then $\delta_\mathcal B(n)\preceq \delta_\mathcal A(n)$, that is, $\delta_\mathcal B(n)\leq C\delta_\mathcal{A}(Cn+C)+Cn+C$ for some constant $C$.
\end{thm}

We now prove the non-automaticity of $G_1$. Note that $G_1$ has
a presentation
\begin{equation}\label{eq:P1}
 \mathcal P_1:= \left\langle a,c,t\ \middle|\ t^{-1}at=ac,\quad t^{-1}ct=c^2 \right\rangle .
\end{equation}
Its Baumslag-Solitar group retract $\BS(1,2)$ has the following presentation 
$$\mathcal{B}_1:=\langle c,t\mid t^{-1}ct=c^2\rangle.$$
Hence, we can easily obtain the following proposition.

\begin{prop}\label{prop:nonautomatic}
The group $G_1$ is not automatic.
\end{prop}
\begin{proof}
Assume $G_1$ is automatic. By Theorem \ref{thm:automatic-quadratic}, there exists a real number $C>0$ such that $\delta_{\mathcal{P}_1}(n)\leq Cn^2+C$ for any $n\geq 1$. By Proposition \ref{prop:retract} and Theorem \ref{thm: Dehn function of retract}, we have $\delta_{\mathcal{B}_1}(n)\preceq \delta_{\mathcal{P}_1}(n)$. Then, there exists another real number $D>0$ such that $$\delta_{\mathcal B_1}(n)\leq D\delta_{\mathcal{P}_1}(Dn+D)+Dn+D\leq DC(Dn+D)^2+DC+Dn+D,$$
which contradicts Theorem \ref{thm: Gersten}. Therefore, $G_1$ is not automatic.
\end{proof}

\section{Proof of the main conclusions and further questions}\label{sect. proof and question}

\begin{proof}[\textbf{Proof of Theorem~\ref{thm:main}}]
Proposition~\ref{prop:retract} proves the retract assertion. Propositions~\ref{prop:nonhopfian} and \ref{prop:nonautomatic} prove, respectively, that $G_2$ is not Hopfian and not residually finite, and that $G_1$ is not automatic.
\end{proof}

Finally, we propose the following question based on Baumslag's question.
\begin{ques}\label{ques. without BS subgp}
    Let $G=F/\ncl{r}$ be a one-relator group with the relator $r=[u,v], ~(u,v\in F)$ and without Baumslag-Solitar subgroups, where $F$ is a finitely generated free group.
    \begin{enumerate}
        \item Is $G$ Hopfian?
        \item Is $G$ residually finite?
        \item Is $G$ automatic?
    \end{enumerate}
\end{ques}

Since all hyperbolic groups are Hopfian \cite{Se99, RW19, FS23} and automatic \cite{Ep92}, items (1) and (3) of Question \ref{ques. without BS subgp} naturally hold if the following conjecture holds. 

\begin{conj}[\cite{Ge92, AG99}, Gersten]\label{conj. Gersten}
    Is every one-relator group without Baumslag-Solitar subgroups hyperbolic?
\end{conj}

\vspace{0.4cm}
\noindent\textbf{AI disclosure.} ChatGPT assisted in constructing the examples. The authors have checked all the arguments, reviewed and edited the content as needed, and take full responsibility for the content of the paper.


\begin{thebibliography}{dd}
\bibitem{AG99} D. Allcock and S.~M. Gersten, \emph{A homological characterization of hyperbolic groups}, Invent. Math. {\bf 135} (1999), no.~3, 723--742.

\bibitem{BMS93} G. Baumslag, C.~F. Miller III and H. Short, \emph{Isoperimetric inequalities and the homology of groups}, Invent. Math. {\bf 113} (1993), no.~3, 531--560.

\bibitem{BMS98} G. Baumslag, A.~G. Myasnikov and V. Shpilrain, ``Open problems in combinatorial group theory'', in {\it Groups, languages and geometry (South Hadley, MA, 1998)}, 1--27, Contemp. Math., 250, Amer. Math. Soc., Providence, RI.


\bibitem{Br63} J.~L. Britton, \emph{The word problem}, Ann. of Math. (2) {\bf 77} (1963), 16--32.

\bibitem{Ep92} D.~B.~A. Epstein et al., \emph{Word processing in groups}, Jones and Bartlett, Boston, MA, 1992.

\bibitem{FS23} K. Fujiwara and Z. Sela, \emph{The rates of growth in a hyperbolic group}, Invent. Math., {\bf 233} (3)(2023), 1427--1470.

\bibitem{Ge89} S.~M. Gersten, ``Dehn functions and $l_1$-norms of finite presentations'', in {\it Algorithms and classification in combinatorial group theory (Berkeley, CA, 1989)}, 195--224, Math. Sci. Res. Inst. Publ., 23, Springer, New York.

\bibitem{Ge92} S.~M. Gersten, ``Problems on automatic groups'', in \emph{Algorithms and classification in combinatorial group theory (Berkeley, CA, 1989)}, 225--232, Math. Sci. Res. Inst. Publ., 23, Springer, New York.

\bibitem{Gr91} M. Gromov, ``Asymptotic invariants of infinite groups'', in {\it Geometric group theory, Vol.\ 2 (Sussex, 1991)}, 1--295, London Math. Soc. Lecture Note Ser., 182, Cambridge Univ. Press, Cambridge.

\bibitem{JL25} A. Jaikin-Zapirain and M. Linton, \emph{On the coherence of one-relator groups and their group algebras}, Ann. of Math. (2) {\bf 201} (2025), no.~3, 909--959.

\bibitem{LW20} L. Louder and H. Wilton, \emph{One-relator groups with torsion are coherent}, Math. Res. Lett. {\bf 27} (2020), no.~5, 1499--1511.

\bibitem{LW24} L. Louder and H. Wilton, \emph{Uniform negative immersions and the coherence of one-relator groups}, Invent. Math. {\bf 236} (2024), no.~2, 673--712.

\bibitem{LS77} R.~C. Lyndon and P.~E. Schupp, \emph{Combinatorial group theory}, Ergebnisse der Mathematik und ihrer Grenzgebiete, Band 89, Springer, Berlin-New York, 1977.

\bibitem{Ma30} W. Magnus, \emph{\"Uber diskontinuierliche Gruppen mit einer definierenden Relation. (Der Freiheitssatz)}, J. Reine Angew. Math. {\bf 163} (1930), 141--165.

\bibitem{Ma32} W. Magnus, \emph{Das Identit\"atsproblem f\"ur Gruppen mit einer definierenden Relation}, Math. Ann. {\bf 106} (1932), no.~1, 295--307.

\bibitem{Ma40} A.~I. Mal'cev, \emph{On isomorphic matrix representations of infinite groups}, Rec. Math. [Mat. Sbornik] N.S. {\bf 8/50} (1940), 405--422.

\bibitem{Ne68} B.~B. Newman, \emph{Some results on one-relator groups}, Bull. Amer. Math. Soc. {\bf 74} (1968), 568--571.

\bibitem{Ol95} A.~Y. Olshanskii, \emph{$\mathrm{SQ}$-universality of hyperbolic groups}, Sb. Math. {\bf 186} (1995), no.~8, 1199--1211; translated from Mat. Sb. {\bf 186} (1995), no.~8, 119--132.

\bibitem{RW19} R. Weidmann and C. Reinfeldt, \emph{Makanin-Razborov diagrams for hyperbolic groups}, Ann. Math. Blaise Pascal {\bf 26} (2019), no.~2, 119--208.

\bibitem{Se99} Z. Sela, \emph{Endomorphisms of hyperbolic groups. I. The Hopf property}, Topology {\bf 38} (1999), no.~2, 301--321.

\bibitem{Wi21} D.~T. Wise, \emph{The structure of groups with a quasiconvex hierarchy}, Ann. of Math. Stud., 209, Princeton Univ. Press, Princeton, NJ, 2021, x+357 pp.
\end{thebibliography}
\end{document}